\documentclass[
submission
]{article}

\usepackage[left=3cm,top=3cm,right=3cm,bottom=3cm]{geometry}
\usepackage[utf8]{inputenc}
\usepackage{subfigure}
\usepackage{amsmath,amssymb,enumerate}
\newcommand{\by}{\hspace{0.02in}{\scriptstyle{\square}}\hspace{0.02in}}
\newtheorem{theorem}{Theorem}
\newtheorem{lemma}[theorem]{Lemma}
\newtheorem{definition}[theorem]{Definition}
\newtheorem{corollary}[theorem]{Corollary}
\newcommand{\slceil}{\lceil\hspace{-0.013in}}
\newcommand{\srceil}{\hspace{-0.013in}\rceil}
\newcommand{\slfloor}{\lfloor\hspace{-0.013in}}
\newcommand{\srfloor}{\hspace{-0.013in}\rfloor}
\newcommand{\qed}{\ \hfill \rule{1ex}{1ex}\\} 
\newenvironment{proof}{{\noindent \bf Proof}: }{\qed}
\usepackage[round]{natbib}

\begin{document}

\author{N.E. Clarke\thanks{Department of Mathematics and Statistics, Acadia University; Support by NSERC 2015-06258}
  \and M.E. Messinger\thanks{Department of Mathematics and Computer Science, Mount Allison University; Supported by NSERC 256119-2011}
  \and G. Power\thanks{Department of Mathematics and Computer Science, Mount Allison University; Supported by Mount Allison University and NSERC USRA 2015}}
\title{Bounding the search number of graph products}

\maketitle



\begin{abstract} In this paper, we provide results for the search number of the Cartesian product of graphs.  We consider graphs on opposing ends of the spectrum: paths and cliques.  Our main result determines the pathwidth of the product of cliques and provides a lower bound for the search number of the product of cliques.  A consequence of this result is a bound for the search number of arbitrary graphs $G$ and $H$ based on their respective clique numbers.\end{abstract}

keywords: graph searching, sweeping, pathwidth

\section{Introduction}\label{sec:intro}

Imagine that a security system has indicated the existence of a camouflaged, mobile intruder in some physical or computer network. How can a set of guards, or  \emph{searchers}, locate this intruder?  Such a question can be considered using a \emph{graph searching} model. In this type of model, an intruder can, at any time, move infinitely fast from vertex $u$ to vertex $v$ along any path that contains no searchers.  To search a graph, it is necessary to formulate and execute a \emph{search strategy}: a sequence of actions designed so that, upon their completion, all edges (and therefore vertices) of the graph have been \emph{cleared} of the invisible intruder.  In such strategies, three actions are permitted and each action may occur multiple times:

{\tiny{$\bullet$}} place a searcher on a vertex; 

{\tiny{$\bullet$}} move a single searcher along an edge $uv$, starting at $u$ and ending at $v$;

{\tiny{$\bullet$}} remove a searcher from a vertex.\vspace{0.05in}

An edge $uv$ can be cleared of the invisible intruder in one of two ways:   (i) at least two searchers are located at vertex $u$ and one of these searchers traverses $uv$ to vertex $v$; (ii) at least one searcher is located at $u$, all edges incident with $u$, other than $uv$, have already been cleared of the intruder, and the searcher traverses the edge $uv$ to vertex $v$.  Naturally, the fundamental question is: what is the fewest number of searchers for which a search strategy exists?  Using the terminology of \cite{YDA}, we call this parameter the \emph{search number} of $G$ and denote it $s(G)$.   This parameter has also been referred to as the edge-search number $es(G)$ (see \cite{GM}, for example) and the sweep number $sw(G)$ (see \cite{Alspach}, for example).  In the literature, searching has been related to pebbling and thus to computer memory usage; it also has applications to assuring privacy when using bugged channels, to VLSI circuit design, and to clearing networks with brushes (see~\cite{Alspach,FL,FGY,KP,MNP,YDA}). The field of graph searching is rapidly expanding and in recent years new models, motivated by applications and foundational issues in computer science, have appeared.

Although the associated decision problem is NP-complete (see~\cite{MHGJP}), the search number is known for many classes of graphs and bounds exist for graphs with particular properties (see  \cite{Alspach,Thesis,YDA}, for example). However, very little is known about the search number of Cartesian products.  \cite{Tosic} provided an upper bound for the search number of the Cartesian product of graphs $G$ and $H$ based on the respective cardinalities and search numbers of $G$ and $H$. \cite{Kinnersley} showed $pw(G) = vs(G)$, where $pw(G)$ denotes the pathwidth (defined below) and $vs(G)$ denotes the vertex separation number of a graph $G$. \cite{JAEllis} showed $vs(G) \leq s(G) \leq vs(G)+1$.  For the Cartesian product $G \by H$ of $G$ and $H$, these results imply \begin{equation}\label{exx}pw(G \by H) \leq s(G \by H) \leq pw(G \by H)+2.\end{equation}  However, as the associated decision problem for pathwidth is NP-complete, the lower bound is not necessarily useful in practice.  

In this paper, we consider input graphs at opposing ends of the spectrum: paths and cliques.  In Section~\ref{sec:path}, we determine $s(P_m \by P_n)$ and $s(K_m \by P_n)$.  In Section~\ref{sec:pw}, we determine $pw(K_m \by K_n)$ and exploit the relationship between the search number and pathwidth to show \begin{equation}\label{ee}s(G \by H) \geq s(K_m \by K_n) \geq pw(K_m \by K_n)  =\begin{cases} \frac{m}{2}n + \frac{m}{2} -1 & \text{\emph{if $m$ even}} \\ \lceil \frac{m}{2} \rceil n -1 & \text{\emph{if $m$ odd}}\end{cases}\end{equation} where $m$, $n$ are the clique numbers of $G$, $H$, respectively.  Inequality (\ref{ee}) is given by Corollary~\ref{end} and results from applying Corollary~\ref{cor:1}, Lemma~~\ref{lem:x}, and Corollary~\ref{cor:paths}.  

To conclude this section, we define the pathwidth  of a graph $G$ and state a simple, but useful, lemma.

\begin{definition}\label{defn}A \emph{path decomposition} of a graph $G$ is a sequence of subsets of vertices $(B_1,B_2,\dots,B_r)$ such that \begin{enumerate}[(i)] \item $\bigcup_{1 \leq i \leq r} B_i = V(G)$;\label{i}

\item For all edges $vw \in E(G)$, $\exists~i \in \{1,2,\dots, r\}$ with $v \in B_i$ and $w \in B_i$;\label{ii}

\item For all $i,j,k \in \{1,2,\dots,r\}$, if $i \leq j \leq k$ then $B_i \cap B_k \subseteq B_j$.\label{iii}\end{enumerate}

The \textbf{width} of a path decomposition $(B_1,B_2,\dots,B_r)$ is $ \max_{1 \leq i \leq r} |B_i|-1$, and the \textup{pathwidth} of $G$, denoted $pw(G)$, is the minimum width over all possible path decompositions of $G$.  \end{definition} 

See the survey by~\cite{Bod} for more on pathwidth; the convention is to refer to subsets $B_1,B_2,\dots,B_r$ as \emph{bags}.  It can easily be seen that an equivalent statement of~(\ref{iii}) is: for each $v \in V(G)$, the set of bags $\{B_i~|~v \in B_i$ and $1 \leq i \leq r\}$ must form a subpath in the decomposition (the important point being that the subpath is, by definition, connected).  To avoid confusion between a path of vertices in a graph and a path of bags in a path decomposition, we will refer to a path of bags in a path decomposition as a \emph{bag-path}.  The next result will be used in Section~\ref{sec:pw} with respect to the product of cliques.  Though the original results are stated for tree decompositions, they obviously apply to path decompositions.  A short proof of the result for tree decompositions exists in~\cite{Bod2}, but the authors state that earlier proofs exist in~\cite{BodClique,German}.

\begin{lemma}\label{clique} Consider a path decomposition $(B_1,B_2,\dots,B_r)$ of  graph $G$, for some positive integer $r$.  Let $W \subseteq V(G)$ be a clique in $G$.  Then $W \subseteq B_i$, for some $1 \leq i \leq r$.   \end{lemma}


\section{Search Number of $P_m \by P_n$ and $P_m \by K_n$}\label{sec:path}


\cite{Ellis} proved that  for $m \geq n$, $pw(P_m \by P_n) = n$ which by Inequality~(\ref{exx}) implies $s(P_m \by P_n) \in \{n,n+1,n+2\}$.  In this section, we show $s(P_m \by P_n)=n+1$ for $m \geq n$.  The notion of a search strategy was described in Section~\ref{sec:intro} as a sequence of actions designed so that once completed, all edges (and therefore vertices) of the graph have been cleared of the invisible intruder.  However, if a search strategy exists for a connected graph, once every searcher has been placed on the graph,  only the action of moving a searcher along an edge is required for the remainder of the search strategy ({\it i.e.} instead of subsequently removing a searcher from a vertex $x$ and placing it on a vertex $y$, the searcher could move along a path from $x$ to $y$).  Thus, if a search strategy exists for a connected graph, then the graph can be cleared by placing the searchers at a set of vertices and then, at each time step, moving one searcher along an edge.  This approach is sometimes called internal searching in the literature and we use it in the proof of Lemma~\ref{lower}.  We note that during the search strategy, recontamination of cleared edges may occur: the process is not necessarily monotonic.  Additionally, at a given time step, any edge that is not clear is considered to be \emph{dirty} and, when recontamination occurs, it occurs instantly.

\begin{lemma}\label{lower} For $n \geq 3$, $s(P_n \by P_n) \geq  n+1$.\end{lemma}

\begin{proof} Let $n \geq 3$ and label the vertices of $P_n \by P_n$ as $v_{i,j}$ for $1 \leq i, j \leq n$. For a contradiction, assume there exists a search strategy for $P_n \by P_n$ that uses $n$ searchers.  Let $R_i$ be the subgraph induced by $\{v_{i,1},v_{i,2},\dots,v_{i,n}\}$ and $C_j$ be the subgraph induced by $\{v_{1,j},v_{2,j},\dots,v_{n,j}\}$; we  informally refer to the subgraphs $R_i$ and $C_j$ as \emph{row $i$} and \emph{column $j$}, respectively.  Let $t$ be the last time step for which \vspace{0.08in}

\noindent (i) at the end of step $t-1$, at least one edge of $R_i$, $C_i$ is dirty for all $i \in [n]$, and

\noindent (ii) at the end of step $t$, every edge of $C_k$ is clear for some $k \in [n]$.  \vspace{0.08in}

\noindent Certainly (i) and (ii) must both occur at some step $t$ in order for there to exist a search strategy of $P_n \by P_n$.   Suppose that for some $x \in [n]$, $R_x$ does not contain a searcher at the end of step $t$.  Then $v_{x,k} = R_x \cap C_k$ is incident with a dirty edge of $R_x$ and a clear edge of $C_k$, which is instantly recontaminated, contradicting (i). Therefore, at the end of step $t$, every row contains at least one searcher. 

From (i) and (ii), we conclude that a searcher moves wlog from $v_{i+1,k}$ to $v_{i,k}$ during step $t$ for some $i \in [n-1]$.   If $i > 1$, a searcher must be located at $v_{i,k}$ immediately prior to step $t$ because edge $(v_{i,k},v_{i+1,k})$ was dirty but edge $(v_{i-1,k},v_{i,k})$ was clean.  Therefore, at the end of step $t$, there are two searchers located at $v_{i,k}$ and all other rows contain at least one searcher:  $s(P_n \by P_n) \geq n+1$.  To complete the proof, we assume $i=1$ and let $t'>t$ be the time step during which a second row or column is cleared.\vspace{0.08in}

\noindent \textbf{Claim:} At the end of step $t$, every edge in $R_1$ is dirty; and after step $t$ and before step $t'$, no searcher can move from one row to another.\vspace{0.08in}

Since only $n$ searchers are available, there is exactly one searcher in each row at the end of step $t$.  During step $t$, a searcher moves from $v_{k,2}$ to $v_{k,1}$ and at the end of step $t$, there is exactly one searcher in each row.  Then at the end of step $t-1$, there is no searcher in $R_1$ (else there are $n+1$ searchers) and, by (i), edge $(v_{1,k}v_{2,k})$ is dirty.  Thus, every edge in $R_1$ is dirty at the end of step $t-1$ and also at step $t$.  

Suppose that after step $t$ and before step $t'$, a searcher moves from row $j$ to row $j+1$ or $j-1$.  Then $R_j$ now contains a dirty edge (by (i) and (ii)) but no searcher.  Any clear edges in $R_j$ immediately become recontaminated along with the two edges of $C_k$ incident with $v_{j,k} \in R_j \cap C_k$. \vspace{0.08in}

The Claim has been proven. \vspace{0.08in} 

To conclude the proof, we consider two cases: $k \in \{2,3,\dots, n-1\}$ and $k \in \{1,n\}$. \vspace{0.1in}

Suppose $k \in \{2,3,\dots, n-1\}$.  For $j \in [n]$, let $s_j$ be the searcher in $R_j$ at the end of step $t$.  At the end of step $t$, $s_1$ is located at $v_{1,k}$ and by the Claim, every edge of $R_1$ is dirty.  Since every edge in $R_1$ is dirty, every vertex of $R_2 \backslash \{v_{2,k}\}$ is incident with a dirty edge.  As there is only one searcher in $R_2$, $s_2$ must be located at $v_{2,k}$ (otherwise $C_k$ is recontaminated via $v_{2,k}$).  By repeating this argument, we find that $s_i$ must be located at $v_{j,k}$ for each $j \in [n]$.  Then each searcher is located at a vertex incident with at least two dirty edges, and so no searcher can move at step $t+1$ without said move resulting in recontamination of at least two edges of $C_k$.  Therefore, $s(P_n \by P_n) \geq n+1$. \vspace{0.1in}

Suppose $k \in \{1,n\}$ and wlog assume $k=1$.  Then during step $t$, searcher $s_1$ moves from $v_{2,1}$ to $v_{1,1}$.  By the Claim, at the end of step $t$, edge $(v_{1,2},v_{1,3})$ is dirty. Then adjacent edge $(v_{1,2},v_{2,2})$ is also dirty at the end of step $t-1$.  Thus, after step $t$ and before step $t'$, $s_1$ may move to $v_{1,2}$, but cannot move elsewhere by the Claim (and because $v_{1,2}$ has at least two incident dirty edges). Thus, at the end of step $t'-1$, $s_1$ is located at either $v_{1,1}$ or $v_{1,2}$ and edge $(v_{1,2},v_{2,2})$ is dirty.  Similarly, at the end of step $t'-1$ for $j \in \{2,3,\dots, n-1\}$, if $s_j$ is located at a vertex of $\{v_{j,1},v_{j,2},\dots,v_{j,j}\}$ then edges $(v_{j,j},v_{j,j+1})$ and $(v_{j,j},v_{j-1,j})$ are dirty.  To prevent recontamination of the edges in $C_1$, searcher $s_{j+1}$ must be located at a vertex of $\{v_{j+1,1},v_{j+1,2},\dots,v_{j+1,j+1}\}$ at the end of step $t'-1$.

Note that searcher $s_n$ cannot be located at $v_{n,n}$ at the end of step $t'-1$; otherwise $R_n$ would be clear before step $t'$.  Thus, $s_n$ is located on one of $\{v_{1,n},v_{2,n},\dots,v_{n-1,n}\}$ at the end of step $t'-1$.  As no searcher is located in $C_n$ at the end of step $t'-1$ and $C_n$ contains at least one dirty edge, every edge of $C_n$ is dirty at the end of step $t'-1$.  

For $R_j$ to be clear by the end of step $t'$, some searcher $s_j$ must move from $v_{j,n-1}$ to $v_{j,n}$.  Thus $j= n-1$ or $j=n$ since, for $j < n-1$, $s_j$ cannot be located at $v_{j,n-1}$ at step $t'-1$.   Since edges $(v_{n-1,n-1},v_{n-2,n-1})$ and $(v_{n-1,n-1},v_{n-1,n})$ are both dirty at step $t'-1$, we note that $j \neq n-1$ (otherwise, edges of $C_1$ are recontaminated).  Therefore, at step $t'$, $s_n$ must move from $v_{n,n-1}$ to $v_{n,n}$, and $R_n$ is clear at the end of step $t'$.  This implies that at the end of step $t'-1$, searcher $s_j$ must be located at $v_{j,j}$, for $2 \leq j \leq n-1$ (otherwise, edges of $C_1$ are recontaminated).  Recall that $(v_{j,j},v_{j-1,j})$, $(v_{j,j},v_{j,j+1})$ are both dirty for $j \in \{2,3,\dots, n-1\}$ at the end of step $t'-1$ (and therefore $t'$).  So none of $s_2,s_3,\dots,s_n$ can move at step $t'+1$ without recontamination of some edges of $C_1$.

Note that $s_n$ could move from $v_{n,n}$ to $v_{n-1,n}$ at step $t'+1$ (or $t'+2$).  However, this results in $s_n$ becoming incident with two dirty edges $(v_{n-1,n-1},v_{n-1,n})$, $(v_{n-2,n},v_{n-1,n})$.   If $s_1$ is located at $v_{1,1}$, then $s_1$ can now move to $v_{1,2}$ at step $t'+1$ (or $t'+2$).  However, this results in $s_1$ being incident with two dirty edges and consequently, all searchers are incident with at least two dirty edges.  So no searcher can move after step $t'+2$ without recontaminating $C_1$. Therefore, $s(P_n \by P_n) \geq n+1$.\end{proof}

\begin{lemma}\label{up} For $n \geq 3$ and a connected finite graph $G$, $s(G \by P_n) \leq |V(G|+1$.\end{lemma}

\begin{proof} Let $n \geq 3$ and $G$ be a connected finite graph.  Label the vertex set of $G \by P_n$ as $v_{i,j}$, for $1 \leq i \leq |V(G)|$ and $1 \leq j \leq n$.  Place one searcher on each vertex of $\{v_{i,1}~:~1 \leq i \leq |V(G)|\}$; we will refer to these searchers as ``the first $|V(G)|$ searchers''.  The $|V(G)|+1^{th}$ searcher clears the edges of the subgraph induced by $\{v_{i,1}~:~1 \leq i \leq |V(G)|\}$.  Then the first $|V(G)|$ searchers move from $v_{i,1}$ to $v_{i,2}$ for each $1 \leq i \leq |V(G)|$ and the $|V(G)|+1^{th}$ searcher clears the edges of the subgraph induced by $\{v_{i,2}~:~1 \leq i \leq |V(G)|\}$.  Continuing in this manner, we find $|V(G)|+1$ searchers sufficient to clear $G \by K_n$.\end{proof}

 \cite{YDA} observed that if $H$ is a minor of $G$, then $s(G) \geq s(H)$.  Since $K_m$ is a minor of $K_m \by P_n$, we observe $s(K_m \by P_n) \geq s(K_m) = m+1$.  Let $\alpha = \min\{m,n\}$. As $P_\alpha \by P_\alpha$ is a minor of $P_m \by P_n$, we observe $s(P_m \by P_n) \geq s(P_\alpha \by P_\alpha) = \alpha +1 = \min\{m,n\}+1$ by Lemma~\ref{lower}.  Applying Lemma~\ref{up} to achieve the upper bounds, the following theorem is immediate.

\begin{theorem} \label{thm:pathclique} For $m,n \geq 3$, $s(P_m \by P_n) = \min\{m,n\}+1$ and $s(K_m \by P_n) = m+1$.\end{theorem}


\section{Pathwidth of the Product of Cliques}\label{sec:pw}


With respect to the search number of products of cliques, \cite{YDA} showed that $s(K_n \by K_2) = n+1$ for $n \geq 3$ and that, for $n \geq 1$, $m\geq 2$, \begin{equation}\label{eq}s(K_m \by K_n) \leq n(m-1)+1.\end{equation}  In this section, we improve the above bound by a factor of a half.  To do this, we consider the pathwidth of $K_m \by K_n$.  Robertson and Seymour introduced the concepts of pathwidth (see \cite{RS2}) and treewidth (see \cite{RS}) which played a fundamental role in their work on graph minors.  Pathwidth is of interest to researchers because many intractable problems can be solved efficiently on graphs of bounded pathwidth.  

Let $\omega(G)$, $\omega(H)$ denote the clique numbers of $G$, $H$, respectively.  Using the result of \cite{YDA} that $s(G) \geq s(H)$ when $H$ is a minor of $G$, the following corollary is immediate. 

 \begin{corollary}\label{cor:1} For any graphs $G$ and $H$,  \begin{displaymath}\text{(a)~~} s(G \by H) \geq \max \{ s(G \by K_{\omega(H)}), s(H \by K_{\omega(G)} )\},\end{displaymath} \begin{displaymath}\text{(b)~} s(G \by H) \geq s(K_{\omega(G)} \by K_{\omega(H)}).\end{displaymath} \end{corollary}
 
Corollary~\ref{cor:1} with Inequality~(\ref{exx}) yields the following relationship with pathwidth.

\begin{lemma}\label{lem:x} (a) For any graphs $G$ and $H$, $s(G \by H) \geq pw(K_{\omega(G)} \by K_{\omega(H)})$. \begin{displaymath}\text{(b)~For~} n \geq 1, m \geq 2, pw(K_m \by K_n) \leq s(K_m \by K_n) \leq pw(K_m \by K_n) + 2.\end{displaymath}
\end{lemma}

For Lemma~\ref{lem:x}~(a) to be useful, $pw(K_{\omega(G)} \by K_{\omega(H)})$ must be known.   The remainder of this section is devoted to proving that for $n \geq m \geq 2$, \begin{displaymath}pw(K_m \by K_n) = \begin{cases} \frac{m}{2}n + \frac{m}{2} -1 & \text{ if $m$ even} \\ \lceil \frac{m}{2} \rceil n -1 & \text{ if $m$ odd.}\end{cases}\end{displaymath} We first note that the treewidth of the product of two cliques of order $n \geq 3$ was determined by \cite{Lucena}: $tw(K_n \by K_n) = \frac{n^2}{2}+\frac{n}{2}-1$.  As treewidth forms a lower bound for pathwidth, the result of~\cite{Lucena} provides a lower bound for $pw(K_n \by K_n)$, for $n \geq 3$.  \cite{ST} showed that construction of a bramble of size $k$ proves $tw(G) \geq k-1$ and, to determine the lower bound for $tw(K_n \by K_n)$, \cite{Lucena} constructed a bramble of order $\frac{n^2}{2}+\frac{n}{2}$.  Although it seems a generalization of the bramble construction of~\cite{Lucena} could be used to obtain a lower bound for $tw(K_m \by K_n)$, this would still only yield a lower bound for $pw(K_m \by K_n)$.  Instead, we consider a direct approach to providing a lower bound for $pw(K_m \by K_n)$, without introducing brambles.  In Section~\ref{sec:upper}, we prove the upper bound for $pw(K_m \by K_n)$ and in Section~\ref{sec:end}, we state conclusions and implications of the upper and lower bounds.  

The following notation is used in the remainder of this section: label the vertex set of $K_m \by K_n$ as $v_{i,j}$ for $1 \leq i \leq m$, $1 \leq j \leq n$.   For any $i \in [m]$, the subgraph of $K_m \by K_n$ induced by vertices $\{v_{1,i},v_{2,i},\dots, v_{m,i}\}$ is called an $m$-clique as it is a subgraph isomorphic to $K_m$.  Similarly, for any $j \in [n]$, the subgraph of $K_m \by K_n$ induced by vertices $\{v_{j,1},v_{j,2},\dots,v_{j,n}\}$ is called an $n$-clique.


\subsection{Lower Bound for the Pathwidth of the Product of Cliques}\label{sec:lower}


\begin{lemma}\label{lem:k2} For $n \geq 2$, $pw(K_2 \by K_n) \geq  n$.\end{lemma}

\begin{proof} For a contradiction, suppose $(B_1,B_2,\dots, B_r)$ is a path decomposition where $\max_{1 \leq i \leq r} |B_i| \leq n$ for some $n \geq 2$. By Lemma~\ref{clique}, there exists $i \in [r], j \in [r]$ such that bag $B_i$ contains the $n$-clique  $\{v_{1,1}$, $v_{1,2}$, $\dots,v_{1,n}\}$ and $B_j$ contains the $n$-clique $\{v_{2,1}$, $v_{2,2}$, $\dots, v_{2,n}\}$.  Certainly, $i \neq j$ (else $|B_i| \geq 2n$), so wlog assume $i < j$.  

Let $B_x$ be the lowest-indexed bag that contains a pair of vertices of the form $v_{1,\alpha},v_{2,\alpha}$, for any $\alpha \in [n]$.  Clearly $i < x < j$ (else one of $B_i,B_j$ contains $n+1$ vertices).  As $B_x$  contains at most $n-2$ vertices other than $v_{1,\alpha},v_{2,\alpha}$, we observe $v_{1,\beta} \notin B_x$, for some $\beta \in [n]$.  Therefore, the pair $v_{1,\beta},v_{2,\beta}$ must appear together in a bag with higher index than $B_x$ (by Definition~\ref{defn}(ii), $v_{1,\beta}, v_{2,\beta}$ must appear in some bag together).   But then we do not have a path decomposition as the set of bags containing $v_{1,\beta}$ does not form a bag-path: $v_{1,\beta} \in B_i$, $v_{1,\beta} \notin B_x$, and $v_{i,\beta}$ is in a bag with higher index than $B_x$.\end{proof}

We next prove a simple, but useful, lemma.  

\begin{lemma}\label{lem:subsets}
 Let $S$ be a set containing $m \geq 3$ elements. Consider an ordered partition of $S$ into at least three non-empty subsets, each of which contains strictly fewer than $\lceil \frac{m}{2} \rceil$ elements, and label the subsets of the ordered partition $S_1,S_2,\dots,S_r$, for some integer $r \geq 3$.  Then, for some $t \in \mathbb{N}$, \begin{displaymath}1 \leq |S_1 \cup S_2 \cup \dots S_{t-1}| \leq \Big\lfloor \frac{m}{2}\Big \rfloor, ~~~1 \leq |S_t| \leq  \Big\lfloor \frac{m}{2} \Big\rfloor, ~~~1 \leq |S_{t+1},S_{t+2},\dots,S_r| \leq  \Big\lfloor \frac{m}{2} \Big\rfloor.\end{displaymath} \end{lemma}

\begin{proof}
Let $t$ be the smallest integer for which $|S_1 \cup S_2 \cup \dots \cup S_t| > \lfloor \frac{m}{2} \rfloor.$  Then $1 \leq |S_1 \cup S_2 \cup \dots \cup S_{t-1}| \leq \lfloor \frac{m}{2} \rfloor.$  By the hypothesis, $1 \leq |S_t| < \lceil \frac{m}{2} \rceil$.  As a result, $1 \leq |S_t| \leq \lfloor \frac{m}{2} \rfloor$ as desired.

As $|S_1 \cup S_2 \cup \dots \cup S_t| > \lfloor \frac{m}{2} \rfloor$, we know $|S_{t+1} \cup S_{t+2} \cup \dots \cup S_r| \leq \lfloor \frac{m}{2} \rfloor$.  It remains to show that $1 \leq |S_{t+1} \cup S_{t+2} \cup \dots \cup S_r|$.  If $m$ is even, then $|S_t|\leq \frac{m}{2}-1< \lceil \frac{m}{2}\rceil$, so $|S_1 \cup S_2 \cup \dots \cup S_t| \leq \frac{m}{2} + \frac{m}{2}-1 < m$.  If $m$ is odd, then $|S_t|\leq \lfloor \frac{m}{2} \rfloor$, so $|S_1 \cup S_2 \cup \dots \cup S_t| \leq \lfloor \frac{m}{2} \rfloor + \lfloor \frac{m}{2} \rfloor < m$.  Thus, $1 \leq |S_{t+1} \cup S_{t+2} \cup \dots \cup S_r|$.\end{proof}

In the remaining $2$ proofs of this subsection, we will repeatedly apply the result of Lemma~\ref{clique} to observe that in a path decomposition, every $n$-clique (and $m$-clique) must be contained in some bag.   

\begin{theorem}\label{lem:lowerpw}
For $n \geq m \geq 4$, $pw(K_m \by K_n) \geq \lceil \frac{m}{2} \rceil n - 1$.
\end{theorem}

\begin{proof} For a contradiction, suppose $(B_1$, $\dots$, $B_r)$ is a path decomposition where $\max_{1 \leq i \leq r} |B_i| \leq \lceil \frac{m}{2} \rceil n - 1$.  Let $S$ be the set of $m$  $n$-cliques in $K_m \by K_n$.  Bags $B_1,B_2,\dots,B_r$ form an ordered partition of $S$ into non-empty subsets, each bag containing fewer than $\lceil \frac{m}{2} \rceil$ $n$-cliques (as each bag contains at most $\lceil \frac{m}{2} \rceil n -1$ vertices).  By Lemma~\ref{lem:subsets},  $X = B_1 \cup B_2 \cup \dots \cup B_{t-1}$ contains $i$ $n$-cliques, for some $i \in [ \lfloor \frac{m}{2}\rfloor]$, $B = B_t$ contains $j$ $n$-cliques, for some $j \in [ \lfloor \frac{m}{2}\rfloor]$,  and $Y = B_{t+1} \cup B_{t+2} \cup \dots \cup B_r$ contains $k$ $n$-cliques, for some $k \in [\lfloor \frac{m}{2} \rfloor]$.  

Suppose wlog that $i \geq k$ and pair each $n$-clique in $Y$ with a distinct $n$-clique in $X$.  For instance, if $\{v_{b,1},v_{b,2},\dots, v_{b,n}\}$ is an $n$-clique in $Y$, it is paired with some $n$-clique $\{v_{a,1},v_{a,2},\dots, v_{a,n}\}$ in $X$.  Every bag on the bag-path between $X$ and $Y$ must contain at least one of $v_{a,\ell}$, $v_{b,\ell}$ for each $\ell \in [n]$ (otherwise we contradict Definition~\ref{defn}(iii)).  Since there are $k$ pairings, there are at least $kn$ vertices in $B$ in addition to the $jn$ vertices from the $j$ $n$-cliques in $B$. So, $|B| \geq (j + k)n = (m-i)n$ as $i+j+k=m$ (the number of $n$-cliques).  Note that $(m-i)n \leq |B| \leq \lceil \frac{m}{2}\rceil n-1$ (the upper bound being the initial hypothesis) implies $i > \lfloor \frac{m}{2} \rfloor$, which contradicts the fact that $i \in [\lfloor \frac{m}{2}\rfloor]$.    Therefore, $B$ contains at least $\lceil\frac{m}{2} \rceil n$ vertices and $pw(K_m \by K_n) \geq \lceil \frac{m}{2} \rceil n - 1$.\end{proof}

Given a minimum width path decomposition $(B_1,B_2,\dots,B_r)$ of graph $G$, the \emph{length} of the decomposition is $r$.  The next result will be used to increase the lower bound of $pw(K_m \by K_n)$ for $m$ even.

\begin{lemma}\label{lem1}
For even $m$ and $n \geq m \geq 4$, suppose $pw(K_m \by K_n) \leq  \frac{m}{2}n + \frac{m}{2}  - 2$ and of the path decompositions of minimum width, let $(B_1,B_2,\dots,B_r)$ be a decomposition of minimum length.  Then for each $i \in [r]$, $B_i$ contains fewer than $\frac{m}{2}$ $n$-cliques.  \end{lemma}

\begin{proof} For $m$ even and $n \geq m \geq 4$, let $(B_1,B_2,\dots,B_r)$ be a minimum length path decomposition for which $\max_{1 \leq i \leq r} |B_i| \leq \frac{m}{2}n+\frac{m}{2}-1$.  We first observe that every bag in the decomposition contains at most $\frac{m}{2}$ $n$-cliques; otherwise, some bag contains at least $ (\frac{m}{2}+1)n = \frac{m}{2}n + n \geq \frac{m}{2}n + m > \frac{m}{2}n+\frac{m}{2}-1$ vertices, which yields a contradiction.  

Next, assume that for some $j \in [r]$, bag $B_j$ contains exactly $\frac{m}{2}$ $n$-cliques.  First, suppose there exists $i < j < k$ such that bags $B_i$, $B_k$ each contain at least one $n$-clique that does not appear in $B_j$.  Let $\{v_{\alpha,1},v_{\alpha,2},\dots,v_{\alpha,n}\}$ be such an $n$-clique in $B_i$ and $\{v_{\beta,1},v_{\beta,2},\dots,v_{\beta,n}\}$ such an $n$-clique in $B_k$.  Then, for each pair $v_{\alpha,s},v_{\beta,s}$ with $s \in [n]$, at least one vertex of the pair must be in $B_j$ (else we contradict Definition~\ref{defn}(ii) and (iii)).  Then $|B_j| \geq \frac{m}{2}n+n>\frac{m}{2}n+\frac{m}{2}-1$ which yields a contradiction.  

Thus, wlog no bag of lower index than $j$ contains an $n$-clique not already contained in $B_j$.  However, then no bag of lower index than $j$ contains an $m$-clique not already contain in $B_j$.  Otherwise, for some $x < j$, $B_x$ contains an $m$-clique and each of these $m$ vertices must appear in a bag as part of its associated $n$-clique.  Thus, each of the $m$ vertices (of the $m$-clique of $B_x$) must appear in $B_j$.  Since exactly $\frac{m}{2}$ of them already appear in $B_j$ in an $n$-clique, this means $|B_j| \geq \frac{m}{2}n+\frac{m}{2}$, which yields a contradiction.

Thus $B_j$ contains $\frac{m}{2}$ $n$-cliques, and no lower-indexed bag contains an $n$-clique or an $m$-clique.  As each vertex must appear in a bag with its associated $m$-clique, every vertex in $B_j$ must appear in $B_{j+1}$.  We now have a contradiction as the minimum width decomposition is not of minimum length.  Consequently, every bag in the decomposition contains strictly fewer than $\frac{m}{2}$ $n$-cliques.\end{proof}

\begin{theorem}\label{thm:low}
For $n \geq m \geq 4$ and $m$ even, $pw(K_m \by K_n) \geq  \frac{m}{2}n + \frac{m}{2}  - 1$.\end{theorem}

\begin{proof} Suppose $n \geq m \geq 4$ and $m$ is even.  For a contradiction, let $(B_1,B_2,\dots,B_r)$ be a minimum length path decomposition for which $\max_{1 \leq i \leq r} |B_i| \leq \frac{m}{2}n+\frac{m}{2}-1$.  As a result of Lemma~\ref{lem1}, we can apply Lemma~\ref{lem:subsets}; let $S$ be the set of $m$  $n$-cliques in $K_m \by K_n$.   Let  $X = B_1 \cup B_2 \cup \dots \cup B_{t-1}$ contain $i$ $n$-cliques for some $i \in [  \frac{m}{2}]$, $B= B_t$ contain $j$ $n$-cliques for some $j \in [  \frac{m}{2}]$,  and $Y = B_{t+1} \cup B_{t+2} \cup \dots \cup B_r$ contain $k$ $n$-cliques for some $k \in [ \frac{m}{2} ]$.   Suppose wlog that $i \geq k$.  

We now show that $X$ and $Y$ each must contain at least one $m$-clique that does not appear in $B$.  To see this, suppose that $X$ contains no $m$-clique: all $m$-cliques appear in $B \cup Y$.  As each vertex must appear in a bag with its associated $m$-clique, it is clear that any vertex of $X$ must also appear in $B$ (else we contradict Definition~\ref{defn}(iii)).  Then $X$ is  unnecessary in the path decomposition, which contradicts the assumption of having a minimum width path decomposition that is of minimum length.  Clearly the same argument ensures $X$ does not contain all the $m$-cliques.  Consequently, $X$, $Y$ each contain at least one $m$-clique.  

We pair the $k$ $n$-cliques in $Y$ with $k$ $n$-cliques in $X$.  If $\{v_{a,1},v_{a,2},\dots, v_{a,n}\}$ in $X$ is paired with $\{v_{b,1},v_{b,2},\dots,v_{b,n}\}$  in $Y$, then $B$ must contain at least one of $v_{a,\ell},v_{b,\ell}$, for each $\ell \in [n]$ (else we contradict Definition~\ref{defn}(iii)).  Thus, $|B| \geq (j+k)n$.  

Recall that $X$, $Y$ must each contain at least one $m$-clique that does not appear in $B$; let  \linebreak $\{v_{1,x},v_{2,x},\dots,v_{m,x}\}$ be such an $m$-clique in $X$ and  $\{v_{1,y},v_{2,y},\dots,v_{m,y}\}$ such an $m$-clique in $Y$.  At least one vertex from each pair $v_{\ell,x},v_{\ell,y}$, for $\ell \in [m]$, must appear in $B$, and at most $j+k$ of these $m$ vertices already appear in $B$.  This leaves an additional $m-(j+k)=i$ vertices.  Thus, $|B| \geq (j+k)n+i = (m-i)n+i = mn-in+i \geq \frac{m}{2}n+\frac{m}{2}$, as $1 \leq i \leq \frac{m}{2}$.  However, this contradicts the initial assumption $pw(K_m \by K_n) \leq \frac{m}{2}n + \frac{m}{2}  - 2$.\end{proof}


\subsection{Upper Bound for the Pathwidth of the Product of Cliques}\label{sec:upper}


We now provide the upper bounds on the pathwidth of the product of cliques.  Theorems~\ref{thm:up} and~\ref{thm:odd} provide upper bounds for even and odd $m$.  

\begin{theorem}\label{thm:up}
For $n \geq m \geq 2$ and $m$ even, $pw(K_m \by K_n) \leq \frac{m}{2} n + \frac{m}{2} - 1$.
\end{theorem}

\begin{proof} For $k \in [n]$, let \begin{displaymath}B_k = \bigcup_{i = k}^n \{v_{1,i},v_{2,i},\dots,v_{\frac{m}{2}, i} \} ~\cup ~\bigcup_{i=1}^k \{v_{\frac{m}{2}+1,i}, v_{\frac{m}{2} +2,i}, \dots, v_{m,i}\}.\end{displaymath}   Observe that each bag contains $\frac{m}{2} n + \frac{m}{2}$ vertices.  We now verify that  $(B_1,B_2,\dots,B_n)$ is a path decomposition.  Consider arbitrary vertex $v_{x,y} \in V(K_m \by K_n)$.  Clearly $v_{x,y} \in B_y$, so $(B_1,B_2,\dots,B_n)$ satisfies condition (i) of Definition~\ref{defn}.

Let $v_{s,t}$ be a vertex adjacent to $v_{x,y}$.  From the definition of the Cartesian product, either $s=x$ or $t=y$.  If $t=y$ then, by the previous paragraph, $v_{s,t}, v_{x,y} \in B_y$.  If $s=x$ then wlog $1 \leq t < y \leq n$.  If $s = x \geq  \frac{m}{2} +1$, then $v_{x,y}, v_{s,t}$ are both in bag $B_y$ as \begin{displaymath}v_{s,t} \in  \bigcup_{i=1}^y \{v_{\frac{m}{2}+1,i}, v_{\frac{m}{2} +2,i}, \dots, v_{m,i}\} \subseteq B_y.\end{displaymath}  If $s = x \leq  \frac{m}{2} $, then $v_{s,t}$, $v_{x,y}$ are both in bag $B_t$ as \begin{displaymath}v_{x,y} \in \bigcup_{i = t}^n \{v_{1,i},v_{2,i},\dots,v_{\frac{m}{2}, i} \} \subseteq B_t.\end{displaymath}    Thus, $(B_1,B_2,\dots,B_n)$ satisfies condition (ii) of Definition~\ref{defn}.

To verify condition (iii) of Definition~\ref{defn}, we assume  $v_{x,y} \in B_p$, $v_{x,y} \notin B_q$, and $v_{x,y} \in B_r$, for $1 \leq p < q < r \leq n$, and seek a contradiction.  As $p < q < n$, if $x \geq \frac{m}{2}+1$, then  \begin{displaymath}v_{x,y} \in \bigcup_{i=1}^p \{v_{\frac{m}{2}+1,i}, v_{\frac{m}{2}+2,i}, \dots, v_{m,i}\} \subseteq B_p \text{ ~implies~ } v_{x,y} \in \bigcup_{i=1}^q \{v_{\frac{m}{2}+1,i}, v_{\frac{m}{2}+2,i}, \dots, v_{m,i}\} \subseteq B_q\end{displaymath} and a contradiction is obtained. As $q < r \leq n$, if $x \leq  \frac{m}{2}$, then \begin{displaymath}v_{x,y} \notin \bigcup_{i = q}^n \{v_{1,i},v_{2,i},\dots,v_{\frac{m}{2}, i} \} \subseteq B_q  \text{ ~implies~ } v_{x,y} \notin  \bigcup_{i = r}^n \{v_{1,i},v_{2,i},\dots,v_{\frac{m}{2}, i} \} \subseteq B_r\end{displaymath} and a contradiction is obtained.  Therefore, $(B_1,B_2,\dots,B_n)$ satisfies condition (iii) of Definition~\ref{defn}.\end{proof}


\begin{theorem}\label{thm:odd}
For $n \geq m \geq 3$ and $m$ odd, $pw(K_m \by K_n) \leq \lceil \frac{m}{2} \rceil n - 1$.
\end{theorem}

\begin{proof} For $1 \leq k \leq \lceil \frac{n}{2} \rceil$, let \begin{equation}\label{1}B_k = \bigcup_{i=1}^{\slfloor \frac{m}{2} \srfloor} \{v_{i,k}, v_{i,k+1},\dots, v_{i,n}\} \cup \bigcup_{i=1}^k \{v_{\slceil \frac{m}{2} \srceil,i}, v_{\slceil \frac{m}{2} \srceil+1,i}, \dots, v_{m,i}\},\end{equation} for $\lceil \frac{n}{2} \rceil + 1 \leq k \leq \lceil \frac{n}{2} \rceil + \lfloor \frac{m}{2} \rfloor$, let 
\begin{equation}\label{2} B_k = \hspace{-2mm} \bigcup_{i = 1}^{k-\slceil\frac{n}{2} \srceil+ \slfloor \frac{m}{2} \srfloor - 1} \hspace{-2mm} \{v_{i, \slceil \frac{n}{2} \srceil + 1}, v_{i, \slceil \frac{n}{2} \srceil + 2}, \dots, v_{i, n}\} \cup~ \bigcup_{i=1}^n  \{v_{k-\slceil \frac{n}{2} \srceil + \slfloor \frac{m}{2} \srfloor, i}\} \cup ~ \hspace{-2mm} \bigcup_{i = k - \slceil \frac{n}{2} \srceil + \slceil \frac{m}{2} \srceil}^{m} \hspace{-2mm} \{v_{i, 1}, v_{i, 2}, \dots v_{i, \slceil \frac{n}{2} \srceil}\}\end{equation} and for $k = \lceil \frac{n}{2} \rceil + \lceil \frac{m}{2} \rceil$, let \begin{equation}\label{3}B_k = \bigcup_{i = 1}^{m - 1} \{v_{i, \slceil \frac{n}{2} \srceil + 1}, v_{i, \slceil \frac{n}{2} \srceil + 2}, \dots, v_{i, n}\} \cup \{v_{m, 1}, v_{m, 2}, \dots, v_{m, n}\}.\end{equation}  We now verify that $(B_1,B_2,\dots,B_{\slceil \frac{n}{2}\srceil + \slceil \frac{m}{2} \srceil})$ is, in fact, a path decomposition.  

Consider an arbitrary vertex $v_{x,y} \in V(K_m \by K_n)$.  If  $1 \leq y \leq \lceil \frac{n}{2} \rceil$ then from (\ref{1}), $v_{x,y} \in B_y.$   If  $\lceil \frac{n}{2}\rceil +1 \leq y \leq n$ then from~(\ref{3}), $v_{x,y} \in  B_{\slceil \frac{n}{2}\srceil + \slceil \frac{m}{2} \srceil}.$ Thus, $(B_1,B_2,\dots,B_{\slceil \frac{n}{2}\srceil + \slceil \frac{m}{2} \srceil})$ satisfies condition (i) of Definition~\ref{defn}. 

Let $v_{s,t}$ be a vertex adjacent to $v_{x,y}$.  From the definition of the Cartesian product, either $s=x$ or $t=y$.  First, suppose $t=y$ and wlog $s<x$.  If $1 \leq y \leq \lceil \frac{n}{2} \rceil$ then from~(\ref{1}), $v_{s,t}$, $v_{x,y} \in B_y$. If $\lceil \frac{n}{2} \rceil +1 \leq y \leq n$, then $v_{s,y}, v_{x,y} \in B_{\slceil \frac{n}{2}\srceil+\slceil \frac{m}{2} \srceil}$.   Second, suppose $s=x$ and wlog $t < y$.   If $1 \leq t \leq \lceil \frac{n}{2}\rceil$ and $1 \leq x \leq \lfloor \frac{m}{2} \rfloor$, then by (\ref{1}), $$v_{x,y},v_{x,t} \in \bigcup_{i=1}^{\slfloor \frac{m}{2}\srfloor} \{v_{i,t},v_{i,t+1},\dots, v_{i,n}\} \subseteq B_t.$$  If $1 \leq t \leq \lceil \frac{n}{2}\rceil$ and $\lfloor \frac{m}{2}\rfloor + 1 \leq x < m$, then observe $\lceil \frac{n}{2}\rceil+1 \leq x+ \lceil \frac{n}{2}\rceil + \lfloor \frac{m}{2}\rfloor < \lceil \frac{n}{2}\rceil + \lceil \frac{m}{2} \rceil$.  So by~(\ref{2}), $$v_{x,y},v_{x,t} \in \bigcup_{i=1}^n \{v_{x,i}\} \subseteq B_{x +\slceil \frac{n}{2} \srceil-\slfloor \frac{m}{2}\srfloor}.$$  If $\lceil \frac{n}{2}\rceil \leq t \leq n$ and $1 \leq x < m$ then by~(\ref{3}), $$v_{x,y},v_{x,t} \in \bigcup_{i=1}^{m-1} \{v_{i,\slceil \frac{n}{2} \srceil +1}, v_{i,\slceil \frac{n}{2}\srceil+2},\dots,v_{i,n}\} \subseteq B_{\slceil \frac{n}{2}\srceil+\slceil \frac{m}{2}\srceil}.$$  Finally, if $x=m$ then by~(\ref{3}), $$v_{x,y},v_{x,t} \in \{v_{m,1},v_{m,2},\dots,v_{m,n}\} \subseteq B_{\slceil \frac{n}{2} \srceil+\slceil \frac{m}{2} \srceil}.$$

To verify condition (iii) of Definition~\ref{defn}, we assume $v_{x,y} \in B_p$, $v_{x,y} \notin B_{p+1}$, $v_{x,y} \in B_r$, for $1 \leq p < p+1<r \leq \lceil \frac{n}{2}\rceil + \lceil \frac{m}{2}\rceil$, and seek a contradiction.\vspace{0.1in}

(a) Suppose $1 \leq p < p+1 \leq \lceil \frac{n}{2}\rceil$.  Then by~(\ref{1}), \begin{displaymath}v_{x,y} \in B_p = \bigcup_{i=1}^{\slfloor \frac{m}{2} \srfloor} \{v_{i,p}, v_{i,p+1},\dots, v_{i,n}\} \cup \bigcup_{j=1}^p \{v_{\slceil \frac{m}{2} \srceil,j}, v_{\slceil \frac{m}{2} \srceil+1,j}, \dots, v_{m,j}\},\end{displaymath} \begin{displaymath}v_{x,y} \notin B_{p+1} = \bigcup_{i=1}^{\slfloor \frac{m}{2} \srfloor} \{v_{i,p+1}, v_{i,p+2},\dots, v_{i,n}\} \cup \bigcup_{i=1}^{p+1} \{v_{\slceil \frac{m}{2} \srceil,i}, v_{\slceil \frac{m}{2} \srceil+1,i}, \dots, v_{m,i}\}.\end{displaymath}  Clearly, the vertices in $B_p$ that are not in $B_{p+1}$ are simply $\{v_{1,p},v_{2,p},\dots,v_{\slfloor \frac{m}{2}\srfloor,p}\}$. It is easy to see for any choice of $r$ where $1 \leq p<p+1<r \leq \lceil\frac{n}{2}\rceil + \lceil \frac{m}{2}\rceil$, \begin{equation}\label{ww}\{v_{1,p},v_{2,p},\dots,v_{\slfloor \frac{m}{2}\srfloor,p}\} \cap B_r = \emptyset.\end{equation}  In particular, we note that if $\lceil \frac{n}{2}\rceil+1 \leq r \leq \lceil \frac{n}{2}\rceil+\lfloor\frac{m}{2}\rfloor$,~(\ref{ww}) follows because $k-\lceil \frac{n}{2}\rceil+\lfloor\frac{m}{2}\rfloor \geq \lceil \frac{m}{2}\rceil$.\vspace{0.1in}

(b) Suppose $1 \leq p \leq \lceil \frac{n}{2}\rceil < p+1 < \lceil \frac{n}{2}\rceil + \lceil \frac{m}{2}\rceil$.  Then we may consider $p=\lceil \frac{n}{2} \rceil$ and $v_{x,y} \in B_{\slceil \frac{n}{2}\srceil}$, but $v_{x,y} \notin B_{\slceil \frac{n}{2}\srceil+1}$.  In this case (using~(\ref{1}) and~(\ref{2})), the vertices in $B_{\slceil \frac{n}{2}\srceil}$ that are not in $B_{\slceil \frac{n}{2}\srceil +1}$ are simply $\{v_{1,\slceil \frac{n}{2} \srceil}, v_{2,\slceil \frac{n}{2} \srceil}, \dots, v_{\slfloor \frac{m}{2} \srfloor, \slceil \frac{n}{2} \srceil}\}$.  Clearly by~(\ref{3}) \begin{displaymath}\{v_{1,\slceil \frac{n}{2} \srceil}, v_{2,\slceil \frac{n}{2} \srceil}, \dots, v_{\slfloor \frac{m}{2} \srfloor, \slceil \frac{n}{2} \srceil}\} \cap B_{\slceil \frac{n}{2} \srceil+\slceil \frac{m}{2}\srceil} = \emptyset\end{displaymath} so $r \neq \rceil \frac{n}{2}\rceil + \lceil \frac{m}{2}\rceil$.  Therefore, $B_r$ is given by~(\ref{2}).  However, as $r > p+1 \geq \lceil \frac{n}{2}\rceil+1$, we know $r \geq \lceil \frac{n}{2}\rceil+2$.  This implies $r - \lceil \frac{n}{2}\rceil + \lfloor \frac{m}{2}\rfloor > \lfloor \frac{m}{2}\rfloor$ and as a result, \begin{displaymath}\{v_{1,\slceil \frac{n}{2} \srceil}, v_{2,\slceil \frac{n}{2} \srceil}, \dots, v_{\slfloor \frac{m}{2} \srfloor, \slceil \frac{n}{2} \srceil}\} \cap B_r = \emptyset.\end{displaymath}\vspace{0.1in}

(c) Suppose $\lceil \frac{n}{2}\rceil +1 \leq p < p+1< \lceil \frac{n}{2}\rceil + \lceil \frac{m}{2}\rceil$. Then \begin{displaymath}B_p = \hspace{-2.5mm}\bigcup_{i = 1}^{p-\slceil\frac{n}{2} \srceil+ \slfloor \frac{m}{2} \srfloor - 1} \hspace{-2.5mm} \{v_{i, \slceil \frac{n}{2} \srceil + 1}, v_{i, \slceil \frac{n}{2} \srceil + 2}, \dots, v_{i, n}\} \cup \bigcup_{i=1}^n \{v_{p-\slceil \frac{n}{2} \srceil + \slfloor \frac{m}{2} \srfloor, i}\} \cup ~ \hspace{-1mm} \bigcup_{i = p - \slceil \frac{n}{2} \srceil + \slceil \frac{m}{2} \srceil}^{m} \hspace{-1mm} \{v_{i, 1}, v_{i, 2}, \dots v_{i, \slceil \frac{n}{2} \srceil}\},\end{displaymath}   \begin{displaymath}B_{p+1} = \hspace{-2mm} \bigcup_{i = 1}^{p-\slceil\frac{n}{2} \srceil+ \slfloor \frac{m}{2} \srfloor} \hspace{-2mm} \{v_{i, \slceil \frac{n}{2} \srceil + 1}, v_{i, \slceil \frac{n}{2} \srceil + 2}, \dots, v_{i, n}\} \cup \bigcup_{i=1}^n \{v_{p-\slceil \frac{n}{2} \srceil + \slfloor \frac{m}{2} \srfloor+1, i}\} \cup ~\hspace{-5mm} \bigcup_{i = p - \slceil \frac{n}{2} \srceil + \slceil \frac{m}{2} \srceil+1}^{m} \hspace{-5mm} \{v_{i, 1}, v_{i, 2}, \dots v_{i, \slceil \frac{n}{2} \srceil}\}\end{displaymath}  and by inspection, the vertices in $B_p$ that are not in $B_{p+1}$ are simply \begin{displaymath}B_p \backslash B_{p+1} = \bigcup_{i=1}^{\slceil \frac{n}{2}\srceil} \{v_{p-\slceil\frac{n}{2}\srceil+\slfloor \frac{m}{2}\srfloor,i}\}.\end{displaymath}  Clearly by~(\ref{3}), $(B_p \backslash B_{p+1}) \cap B_{\slceil \frac{n}{2}\srceil + \slceil \frac{m}{2}\srceil} = \emptyset$, so $r \neq \lceil \frac{n}{2}\rceil+\lceil \frac{m}{2}\rceil$.  Then by~(\ref{2}), \begin{displaymath} B_r = \hspace{-2mm} \bigcup_{i = 1}^{r-\slceil\frac{n}{2} \srceil+ \slfloor \frac{m}{2} \srfloor - 1} \hspace{-2mm} \{v_{i, \slceil \frac{n}{2} \srceil + 1}, v_{i, \slceil \frac{n}{2} \srceil + 2}, \dots, v_{i, n}\} \cup~ \bigcup_{i=1}^n  \{v_{r-\slceil \frac{n}{2} \srceil + \slfloor \frac{m}{2} \srfloor, i}\} \cup ~ \hspace{-2mm} \bigcup_{i = r - \slceil \frac{n}{2} \srceil + \slceil \frac{m}{2} \srceil}^{m} \hspace{-2mm} \{v_{i, 1}, v_{i, 2}, \dots v_{i, \slceil \frac{n}{2} \srceil}\}.\end{displaymath}  Since $r>p+1$, we know $r-\lceil \frac{n}{2}\rceil + \lceil \frac{m}{2}\rceil > p - \lceil \frac{n}{2}\rceil + \lfloor \frac{m}{2}\rfloor$ and thus, by inspection conclude $(B_p \backslash B_{p+1}) \cap B_r = \emptyset$.  Note that condition (iii) of Definition~\ref{defn} is satisfied.

Therefore, $(B_1,B_2,\dots, B_{\slceil \frac{n}{2}\srceil + \slceil \frac{m}{2}\srceil})$ forms a path decomposition.  As $n \geq m \geq 3$, counting the number of vertices in $B_k$ for~(\ref{1}),~(\ref{2}), and~(\ref{3}) finds $|B_k| \leq \lceil \frac{m}{2}\rceil n$.  \end{proof}


\subsection{The Pathwidth of the Product of Cliques}\label{sec:end}


The following corollary is immediate from Lemma~\ref{lem:k2} and Theorems \ref{lem:lowerpw}, \ref{thm:low}--\ref{thm:odd}.

\begin{corollary}\label{cor:paths} For $n \geq m \geq 2$, \begin{displaymath}pw(K_m \by K_n) = \begin{cases} \frac{m}{2}n + \frac{m}{2} -1 & \text{ if $m$ even} \\ \lceil \frac{m}{2} \rceil n -1 & \text{ if $m$ odd.}\end{cases}\end{displaymath} \end{corollary}

Our final results follow directly from Corollary~\ref{cor:1}(a), Lemma~\ref{lem:x}, and Corollary~\ref{cor:paths}. Corollary~\ref{SearchClique} bounds the search number of the Cartesian product of cliques to within $2$ and improves the bound of~\cite{YDA}, given in Inequality~(\ref{eq}), by half. Corollary~\ref{end} provides the lower bound for the search number of the product of two general graphs $G$ and $H$.
 
  \begin{corollary}\label{SearchClique} For $n \geq m \geq 2$, if $m$ is even, then \begin{displaymath}\frac{m}{2}n+\frac{m}{2}-1 \leq s(K_m \by K_n) \leq \frac{m}{2}n+\frac{m}{2}+1\end{displaymath} and, if $m$ is odd, then \begin{displaymath}\Big\lceil \frac{m}{2}\Big\rceil n-1 \leq s(K_m \by K_n) \leq \Big\lceil \frac{m}{2}\Big\rceil n + 1.\end{displaymath} \end{corollary}

\smallskip

 \begin{corollary}\label{end} For $|V(H)|\geq |V(G)|\geq 4$, where the clique numbers of graphs $G$ and $H$ are $m$ and $n$ respectively, \begin{displaymath}s(G \by H) \geq \begin{cases} \frac{m}{2}n + \frac{m}{2} -1 & \text{ if $m$ even} \\ \lceil \frac{m}{2} \rceil n -1 & \text{ if $m$ odd.}\end{cases}\end{displaymath}  \end{corollary}

\nocite{*}
\bibliographystyle{abbrvnat}
\bibliography{TheBib}
\label{sec:biblio}

\end{document}